\documentclass{amsart}
\usepackage{a4}
\usepackage{stackrel} 
\usepackage{graphicx}
\usepackage{color}
\usepackage{amssymb}
\usepackage[all]{xy}
\usepackage{amsmath}
\usepackage{stmaryrd}
\CompileMatrices

\newtheorem{theorem}{Theorem}[section]
\newtheorem{lemma}[theorem]{Lemma}
\newtheorem{proposition}[theorem]{Proposition}
\newtheorem{corollary}[theorem]{Corollary}

\theoremstyle{definition}

\newtheorem{definition}[theorem]{Definition}
\newtheorem{example}[theorem]{Example}

\newtheorem{problem}[theorem]{Problem}
\newtheorem{construction}[theorem]{Construction}

\newtheorem{remark}[theorem]{Remark}
\theoremstyle{remark}

\numberwithin{equation}{section}
\def\nquotchow{
\!\!\!
\raisebox{0.09cm}{
\scalebox{.47}{
$
\stackbin[{\rm CQ}]{\hphantom{...}\sim}{\hphantom{.}\boldsymbol{/}}
$
}
}
\!
}

\def\nqquot{
\!\!\!
\raisebox{0.09cm}{
\scalebox{.47}{
$
\stackbin[{\rm LQ}]{\hphantom{...}\sim}{\hphantom{.}\boldsymbol{/}}
$
}
}
\!
}

\def\quotchow{
\!\!\!
\raisebox{0.09cm}{
\scalebox{.47}{
$
\stackbin[{\rm CQ}]{\hphantom{...}}{\hphantom{.}\boldsymbol{/}}
$
}
}
\!
}

\def\qquot{
\!\!\!
\raisebox{0.09cm}{
\scalebox{.47}{
$
\stackbin[{\rm LQ}]{\hphantom{...}}{\hphantom{.}\boldsymbol{/}}
$
}
}
\!
}

\def\quotnull{
\!\!\!
\raisebox{0.09cm}{
\scalebox{.47}{
$
\stackbin[\!\!0]{\hphantom{...}}{\hphantom{.}\boldsymbol{/}}
$
}
}
\!
}

\def\quot{/\!\!/}
\def\mal{\! \cdot \!}

\def\rq#1{\widehat{#1}}

\def\bangle#1{\langle #1 \rangle}

\def\KK{{\mathbb K}}
\def\TT{{\mathbb T}}
\def\ZZ{{\mathbb Z}}

\def\QQ{{\mathbb Q}}
\def\PP{{\mathbb P}}

\def\trop{\operatorname{trop}}

\def\Ch{\operatorname{Ch}}
\def\Cl{\operatorname{Cl}}

\def\Mat{{\rm Mat}}

\def\Spec{{\rm Spec}}

\def\cone{{\rm cone}}

\begin{document}

\title[On Chow quotients of torus actions]%
{On Chow quotients of torus actions}

\author[H.~B\"aker]{Hendrik B\"aker}
\address{Mathematisches Institut,
Universit\"at T\"ubingen,
Auf der Morgenstelle 10,
72076 T\"ubingen,
Germany}
\email{hendrik.baeker@uni-tuebingen.de}

\author[J.~Hausen]{J\"urgen Hausen}
\address{Mathematisches Institut,
Universit\"at T\"ubingen,
Auf der Morgenstelle 10,
72076 T\"ubingen,
Germany}
\email{juergen.hausen@uni-tuebingen.de}

\author[S.~Keicher]{Simon Keicher}
\address{Mathematisches Institut, Universit\"at T\"ubingen,
Auf der Morgenstelle 10, 72076 T\"ubingen, Germany}
\email{keicher@mail.mathematik.uni-tuebingen.de}


\begin{abstract}
We consider torus actions on Mori dream spaces
and ask whether the associated Chow quotient 
is again a Mori dream space and, if so, what
does its Cox ring look like.  
We provide general tools for the study of these 
problems and give solutions for $\KK^*$-actions 
on smooth quadrics.
\end{abstract}

\maketitle

\section{Introduction}

Consider an action $G \times X \to X$ of a 
connected linear algebraic group $G$ on a 
projective variety $X$ defined over an 
algebraically closed field $\KK$ of characteristic 
zero.
The Chow quotient is an answer to the problem  
of associating in a canonical way a quotient 
to this action: it is defined as the 
closure of the set of general $G$-orbit closures 
viewed as points in the Chow variety, 
see Section~\ref{sec:chowquots} for more 
background.
The Chow quotient always exists but, in general,
its geometry appears to be not easily accessible.

The perhaps most prominent example is the 
Grothendieck-Knudsen moduli space $\overline{M}_{0,n}$
of stable $n$-pointed curves of genus zero.
Kapranov~\cite{Kap} showed that it arises as the Chow 
quotient of the maximal torus action on the Grassmannian
$G(2,n)$.
The space $\overline{M}_{0,n}$ is intensely studied and 
one of the challenging questions is whether or not 
$\overline{M}_{0,n}$ is a Mori dream space, i.e.~its 
Cox ring is finitely generated, see~\cite{HuKe}.
This is easily seen for $n \le 5$, was proven 
for $n=6$ by Castravet~\cite{Ca} and is open for
$n \ge 7$.
Motivated by this example, we formulate the following.

\begin{problem}
\label{prob:mdschowquot}
Consider the action $T \times X \to X$ 
of an algebraic torus $T$ on a Mori 
dream space $X$ and  the normalization $Y$
of the associated Chow quotient.
\begin{enumerate}
\item
Is $Y$ a Mori dream space?
\item
If $Y$ is a Mori dream space, describe its 
Cox ring.
\end{enumerate}
\end{problem}

The situation is well understood in the case of
subtorus actions on toric varieties~\cite{KaStZe,CrML}. 
There, the normalized Chow quotient is again 
toric and hence a Mori dream space. 
Moreover, the corresponding fan can be computed
and thus the Cox ring of the normalized Chow 
quotient is accessible as well.
Similarly, one may treat subtorus actions on 
rational varieties with a complexity one torus 
action using their recent combinatorial description.
In the present paper, we provide tools 
for a study of the general case. 
For example, our methods allow a complete answer 
to~\ref{prob:mdschowquot}~(i)
in the case of $\KK^*$-actions on smooth projective
quadrics:

\begin{theorem}
\label{mainthm1}
Let $\KK^*$ act on a smooth projective quadric $X$.
Then the associated normalized Chow quotient is 
a Mori dream space.
\end{theorem}

Note that a positive answer to the question~\ref{prob:mdschowquot}~(i) 
in the case of $\KK^*$-actions on arbitrary Mori dream 
spaces will imply a positive answer for all torus actions
on Mori dream spaces:
as we show in Theorem~\ref{thm:reduct}, the normalized 
Chow quotient of a torus action is birationally dominated 
by the space obtained via stepwise taking normalized 
Chow quotients by subtori and thus, if the latter space 
has finitely generated Cox ring, then the normalized 
Chow quotient does so.

\goodbreak

We turn to Problem~\ref{prob:mdschowquot}~(ii).
The motivation to describe the Cox ring is that 
this leads to a systematic approach to the geometry 
of the Chow quotient.
Let us present the results in the case of $\KK^*$-actions
on quadrics. After equivariantly embedding into a 
projective space and applying a suitable linear 
transformation, the smooth projective quadric 
$X$ is of the following shape:
$$
X 
\ = \
V(g_1)
\ \subseteq \ 
\PP_r,
\qquad
g_1  \ = \
\begin{cases}
T_0T_1 + \ldots + T_{r-1}T_r, & r \text{ odd,}
\\
T_0T_1 + \ldots + T_{r-2}T_{r-1} + T_r^2, & r \text{ even,}
\end{cases}
$$
where the $\KK^*$-action is diagonal with weights 
$\zeta_0, \ldots, \zeta_r$ and the defining equation
is of degree zero.
In order to write down the Cox ring of the Chow quotient,
consider the extended weight matrix
\begin{eqnarray*}
Q 
& := & 
\left[
\begin{array}{ccc}
\zeta_0 & \ldots & \zeta_r
\\
1 & \ldots & 1
\end{array}  
\right]
\end{eqnarray*}
where we assume that the columns of $Q$ generate
$\ZZ^2$.
Let $P$ be an integral Gale dual, 
i.e.~an $r-1$ by $r+1$ matrix 
with the row space of $Q$ as kernel. 
Determine the Gelfand-Kapranov-Zelevinsky 
decomposition~$\Sigma$ associated to $P$
and put the primitive generators 
$b_1,\ldots, b_l$ of $\Sigma$
differing from the columns of $P$ as columns into 
a matrix $B$. 
Then there is an integral matrix $A$ such that
$B = P \cdot A$ holds.
Define shifted row sums 
$$
\eta_i  :=   A_{i \, *} + A_{i+1 \, *} + \mu 
\quad \text{ for } i = 0,2, \ldots;
\qquad\quad
\eta_r := 2  A_{r \, *} + \mu, 
\quad\text{ if } r \text{ is even}.
$$
where $\mu$ is the componentwise minimal vector 
such that the entries of the $\eta_i$ are all
nonnegative. Then our result reads as follows.

\begin{theorem}
\label{mainthm2}
In the above setting, assume that any $r$ 
columns of $Q$ generate~$\ZZ^2$ and that 
for odd (even) $r$ there are at least four (three) 
weights $\zeta_i$ of minimal absolute value.
Then the normalized Chow quotient $Y$ 
of the $\KK^*$-action on $X$ has Cox ring 
\begin{eqnarray*}
\mathcal{R}(Y)
& = & 
\KK[T_0,\ldots,T_r,S_1,\ldots, S_l]
\ / \ 
\bangle{g_2}
\end{eqnarray*}
with 
\begin{eqnarray*}
g_2
& := &
\begin{cases} 
T_0T_1S^{\eta_0} 
+
T_2T_3S^{\eta_2} 
+ \ldots + 
T_{r-1}T_r S^{\eta_{r-1}},
&
r \text{ odd},
\\
T_0T_1S^{\eta_0} 
+  \ldots + 
T_{r-2}T_{r-1}S^{\eta_{r-2}} 
+
T_r^2 S^{\eta_{r}},
&
r \text{ even}
\end{cases}
\end{eqnarray*}
graded by $\ZZ^{l+2}$ via assigning to the $i$-th variable
the $i$-th column of a Gale dual of the block matrix
$[P,B]$.
\end{theorem}

The proof of Theorem~\ref{mainthm2}
is performed in Section~\ref{sec:quadrics1}.
Besides the explicit description of the 
rays of the Gelfand-Kapranov-Zelevinsky 
decomposition 
provided in Proposition~\ref{cor:gkzrays}, 
it requires controlling the 
behaviour of the Cox ring under certain 
modifications. 
This technique is of independent interest
and developed in full generality in 
Section~\ref{sec:ambmod1}.
The proof of Theorem~\ref{mainthm1}, given in 
Section~\ref{sec:quadrics2}, uses moreover
methods from tropical geometry: we consider
a ``weak tropical resolution'' of the Chow
quotient, see Construction~\ref{constr:tropresgen},
and provide a reduction principle to divide
out intrinsic torus symmetry,
see Proposition~\ref{prop:tropembmdschar}.
These methods are constructive and will
be used in~\cite{BaHaKe} for a computational 
study of Chow quotients, e.g.~of certain
Grasmannians.

\tableofcontents

\section{Chow quotients and limit quotients}
\label{sec:chowquots}

We present the necessary background and a 
general result for the action of a torus~$T$ 
on a projective (irreducible) variety~$X$.
For the precise definition of the quotients, 
consider more generally the action 
$G \times X \to X$ of any connected 
linear algebraic group~$G$ on a projective 
variety $X$. 
The Chow quotient has been introduced by Kapranov, 
Sturmfels and Zelevinsky in~\cite{KaStZe}. 
Initially, the construction appears to depend on an 
embedding but finally turns out not do so.

\begin{construction}
Suppose that $X$ is a $G$-invariant closed 
subvariety of some projective space.
For a suitable open invariant subset 
$U \subseteq X$, all orbit closures 
$c(x) := \overline{G \mal x}$,
where $x \in U$, have the same dimension 
$k$ and degree $d$.
Thus, each $x \in U$ defines a point 
$c(x) \in \Ch(X)$ in the Chow variety of 
$k$-cycles of degree $d$.
The {\em Chow quotient\/} of the 
$G$-action on $X$ is the closure 
$$
X \quotchow G 
\ := \ 
\overline{\{c(x); \ x \in U\}}
\ \subseteq \
\Ch(X).
$$ 
By the {\em normalized Chow quotient\/} 
we mean the normalization $X \nquotchow G$
of $X \quotchow G$.
With a suitably small chosen $U \subseteq X$,
one obtains a commutative diagram of morphisms 
involving the normalization map:
$$ 
\xymatrix{
& 
U 
\ar[dl]
\ar[dr]
&
\\
X \nquotchow G
\ar[rr]
&&
X \quotchow G.
}
$$
\end{construction}

The limit quotient arises from the variation 
of Mumford's GIT quotients~\cite{Mu}.
Its construction relies on finiteness of the 
number of possible sets of semistable 
points~\cite{DoHu,Thadd}.

\begin{construction}
Suppose that $G$ is reductive.
Let $X_1, \ldots, X_r \subseteq X$ be the 
open sets of semistable points arising from 
$G$-linearized ample line bundles on $X$.
Then, whenever $X_i \subseteq X_j$ holds, 
we have a commutative diagram
$$ 
\xymatrix{
X_i 
\ar[r]
\ar[d]
&
X_j
\ar[d]
\\
X_i \quot G
\ar[r]_{\varphi_{ij}}
&
X_j \quot G
}
$$
where the induced map 
$\varphi_{ij} \colon X_i  \quot G \to X_j  \quot G$ 
of quotients is a dominant projective morphism.
This turns the quotient spaces into a directed
system, the {\em GIT system}.
The associated {\em GIT limit} $Y$, 
i.e.~the inverse limit,
comes with a canonical morphism
$$ 
U 
\ := \ 
\bigcap_{i=1}^r X_i
\quad
\to 
\quad 
Y.
$$
The closure of the image of this morphism 
is denoted by $X \qquot G$ and is called 
the {\em limit quotient\/}. There are canonical 
proper birational morphisms onto the GIT 
quotients:
$$
\pi_i \colon X \qquot G \ \to \ X_i \quot G.
$$
The {\em normalized limit quotient\/} is the 
normalization $X \nqquot G$ of $X \qquot G$.
Suitably shrinking the open set $U \subseteq X$, 
we obtain a commutative diagram involving the 
normalization map:
$$ 
\xymatrix{
& 
U 
\ar[dl]
\ar[dr]
&
\\
X \nqquot G
\ar[rr]
&&
X \qquot G.
}
$$
\end{construction}

Note that, in the literature, $X \qquot G$ is called 
also the ``canonical component'' of the GIT limit,
or even shortly the ``GIT limit''.
Similar to the full inverse limit, the quotient
$X \qquot G$ enjoys a universal property.

\begin{remark}
\label{rem:gitlimunivprop}
Given an irreducible variety $W$ and 
a collection of dominant morphisms 
$\psi_i \colon W \to X_i \quot G$
with $\psi_j = \varphi_{ij} \circ \psi_i$ 
for all $i,j$, there is a unique morphism 
$\psi \colon W \to X \qquot G$ with 
$\psi_i = \pi_i \circ \psi$ for all $i$.
\end{remark}

For a general reductive group action, 
the (normalized) Chow quotient and the
(normalized) limit quotient need not 
coincide. For torus actions, however, 
they do. This statement seems to have 
folklore status; a proof under a certain 
hypothesis can be found in~\cite[Thm.~3.8]{Hu}.
Let us indicate how to deduce it from 
the corresponding statement in the case 
of subtorus actions on projective toric 
varieties obtained in~\cite{KaStZe, CrML}.

We recall the necessary results and concepts
from~\cite{KaStZe, CrML}.
Let $Z$ be a projective toric variety with 
acting torus $T_Z$ and consider the action
of a subtorus $T \subseteq T_Z$.
The toric variety~$Z$ arises from a 
fan $\Sigma$ in some $\ZZ^r$ and 
$T \subseteq T_Z$ corresponds to an
embedding $\ZZ^k \subseteq \ZZ^r$
of a sublattice.
Let $P \colon \ZZ^r \to \ZZ^{r-k}$ 
be the projection.
The {\em quotient fan\/} of $\Sigma$ 
with respect to $P$ is the  
fan in $\ZZ^{r-k}$ with the cones 
$$ 
\tau(v) 
\ := \ 
\bigcap_{\sigma \in \Sigma, v \in P(\sigma)} P(\sigma),
\qquad
v \in \QQ^{r-k}.
$$

\begin{proposition}
\label{prop:chow2GIT}
See~\cite{KaStZe, CrML}.
Consider the toric variety $Z$ arising from 
a fan~$\Sigma$ in~$\ZZ^r$ and the action of 
a subtorus $T \subseteq T_Z$ corresponding
to a sublattice $\ZZ^k \subseteq \ZZ^r$. 
Let~$\Sigma'$ be the quotient fan in $\ZZ^{r-k}$ 
with respect to $\ZZ^r \to \ZZ^{r-k}$ and 
$Z'$ the associated toric variety.
Then we have a commutative diagram
$$ 
\xymatrix@R=15pt{
&
T_Z/T
\ar[dl]
\ar[d]
\ar[dr]
&
\\
Z \nquotchow T
\ar@{<->}[r]_{\quad \cong}
\ar[d]
&
Z'
\ar@{<->}[r]_{\cong \quad}
&
Z \nqquot T
\ar[d]
\\
Z \quotchow
\ar@{<->}[rr]_{\cong \quad}
&
&
Z \qquot T
}
$$
In particular, the (normalized) Chow quotient 
and the (normalized) limit quotient of the 
$T$-action on $Z$ are isomorphic
to each other.
\end{proposition}


We turn to the general case. The result is 
formulated for a projective variety~$X$ 
which is equivariantly embedded into a toric variety 
$Z$. Note that for a normal projective~$X$,
this can always be achieved, even with a 
projective space $Z$.

\begin{proposition}
\label{thm:chow2GIT}
Let $Z$ be a projective toric variety, 
$T \subseteq T_Z$ a subtorus of the big torus 
and $X \subseteq Z$ a closed $T$-invariant 
subvariety intersecting $T_Z$.
Then there is a commutative diagram
$$
\xymatrix@R=10pt{
&
{X \quotchow T}
\ar@{<->}@/_10pc/[dddd]_{\cong}
\ar[rr]^{\rm embedding}
& &
{Z \quotchow T}
\ar@{<->}@/^8pc/[dddd]^{\cong}
\\
&
{X \nquotchow T}
\ar[rr]^{\rm finite}
\ar@{=}[dd]
\ar[u]
& &
{Z \nquotchow T}
\ar@{=}[dd]
\ar[u]
\\
(X \cap T_Z)/T 
\ar[ur]
\ar[dr]
\ar[uur]
\ar[ddr]
\ar@{-}[r]
&
\ar@{-}[rr]
&
&
\ar[r]
&
T_Z/T 
\ar[ul]
\ar[dl]
\ar[uul]
\ar[ddl]
\\
&
{X \nqquot T}
\ar[rr]_{\rm finite}
\ar[d]
& &
{Z \nqquot T}
\ar[d]
\\
&
{X \qquot T}
\ar[rr]_{\rm embedding}
& &
{Z \qquot T}
}
$$
where $X \nquotchow T \to Z \nquotchow T$
and $X \nqquot T \to Z \nqquot T$
normalize the closures of the images of 
$(X \cap T_Z)/T$ 
under the canonical open embeddings of $T_Z/T$.
\end{proposition}

\begin{proof}
The right part of the diagram is 
Proposition~\ref{prop:chow2GIT}.
The closed embedding 
$X \quotchow T \to Z \quotchow T$
exists by the construction of 
the Chow quotient; 
compare also~\cite[Thm.~3.2]{GiML}.

To obtain a morphism $X \qquot T \to Z \qquot T$, 
consider the sets of semistable points
$V_1, \ldots, V_s \subseteq  Z \nquotchow T$
defined by $T$-linearized ample line bundles
on $Z$.
Then the sets $U_i := X \cap V_i$ are
sets of semistable points of the respective 
pullback bundles, see~\cite[Thm.~1.19]{Mu}
and we have induced morphisms $U_i \quot T \to V_i \quot T$.
Since the $U_i \quot T$ form a subsystem 
of the full GIT-system of $X$,
the universal property~\ref{rem:gitlimunivprop}
yields a morphism of the limit quotients
sending $X \qquot T$ birationally onto the 
closure of $(X \cap T_Z)/T$.

Now look at the canonical morphism 
$X \quotchow T \to X \qquot T$ 
provided by~\cite{Kap,Thadd}. It fits 
into the diagram established so far
which in turn implies that 
$X \quotchow T \to X \qquot T$ is an
isomorphism and $X \qquot T \to Z \qquot T$
is an embedding.
Finally, the respective normalizations 
fit into the diagram via their 
universal properties.
\end{proof}

Note that we will only use the part of 
Proposition~\ref{thm:chow2GIT} concerning 
the normalizations. This can be proved by 
similar arguments as above but without using 
the isomorphism $Z \quotchow T \to Z \qquot T$ 
of Proposition~\ref{prop:chow2GIT}.

\begin{corollary}
Let $T \times X \to X$ be the action of a torus $T$
on a normal projective variety $X$.
Then the normalized Chow quotient $X \nquotchow T$ and 
the normalized limit quotient $X \nqquot T$ are isomorphic
to each other.
\end{corollary}

%

The following corollary shows that for torus actions, the limit 
quotient is up to normalization already determined by 
the possible linearizations of a single ample bundle;
a statement which fails in general for other reductive groups, 
compare also~\cite[Remark~0.4.10]{Kap}.

\begin{corollary}
Let $T \times X \to X$ be the action of a torus $T$
on a normal projective variety $X$. 
Then the subsystem of GIT quotients arising from the 
possible $T$-linearizations of a given ample line bundle
$\mathcal{L}$ has 
the same normalized limit quotient as the full system 
of GIT quotients.
\end{corollary}

\begin{proof}
Fix a $T$-linearization of $\mathcal{L}$ and consider 
the $T$-equivariant embedding $X \to \PP_r$ defined 
by the a suitable power of  $\mathcal{L}$.
Then the subsystem of the GIT quotients on $X$ 
arising from other linearizations of $\mathcal{L}$
is induced from the full GIT system on $\PP_r$.
Now apply Proposition~\ref{thm:chow2GIT}.
\end{proof}

We now prove the reduction theorem. It says 
in particular, that the Chow quotient 
of a torus action is birationally dominated 
by an iterated Chow quotient with respect to 
$\KK^*$-actions.

\begin{theorem}
\label{thm:reduct}
Let $T \times X \to X$ be the action of a torus $T$
on a normal projective variety~$X$.
Fix a subtorus $T_0 \subseteq T$ and set 
$T_1 := T/T_0$.
Then we have canonical proper birational morphisms 
$$ 
(X \nquotchow T_0) \nquotchow T_1
\ \to \ 
X \nquotchow T,
\qquad\qquad
(X \nqquot T_0) \nqquot T_1
\ \to \ 
X \nqquot T.
$$
\end{theorem}

\begin{proof}
First consider the case that $T$ is a 
subtorus of the big torus $T_Z$ of a toric 
variety $Z$.
Then the maps $T_Z \to T_Z/T_0 \to T_Z/T$
correspond to lattice homomorphisms 
$\ZZ^r \to \ZZ^{r-k_0}  \to \ZZ^{r-k}$.
The fan $\Sigma$ of $Z$ lives in $\ZZ^r$
and we have the quotient fan $\Sigma_0$ 
of $\Sigma$ with respect to 
$\ZZ^r \to \ZZ^{r-k_0}$.
The quotient fan of $\Sigma_0$ 
with respect to 
$\ZZ^{r-k_0} \to \ZZ^{r-k}$ 
refines the quotient fan of $\Sigma$ 
with respect to $\ZZ^r \to \ZZ^{r-k}$.
Translated to toric varieties, this
means that we have the desired maps
$$
(Z \nquotchow T_0) \nquotchow T_1
\ \to \ 
Z \nquotchow T,
\qquad\qquad
(Z \nqquot T_0) \nqquot T_1
\ \to \ 
Z \nqquot T.
$$

We turn to the general case.
Suitably embedding $X$, we can arrange
the setup of Proposition~\ref{thm:chow2GIT}.
Then we have a finite $T_1$-equivariant map 
$\nu \colon X \nquotchow T_0 \to Z \nquotchow T_0$.
We consider the normalized limit quotient
of the $T_1$-action on $X \nquotchow T_0$.
In a first step, we establish a commutative 
diagram
$$ 
\xymatrix{
&
(X \cap T_Z)/T_0
\ar[dl]
\ar[dr]
&
\\
(X \nquotchow T_0) \qquot T_1 
\ar[rr]
&
&
(Z \nquotchow T_0) \qquot T_1
}
$$
For this, let 
$V_1, \ldots, V_s \subseteq  Z \nquotchow T_0$
be the sets of semistable points arising from
$T_1$-linearized ample line bundles.
Then the inverse images 
$\nu^{-1}(V_i) \subseteq X \nquotchow T_0$
are sets of semistable points of the respective 
pullback bundles, see~\cite[Thm.~1.19]{Mu}.
Note that we have canonical induced maps
$$
\nu^{-1}(V_i) \quot T_1 
\ \to \ 
V_i \quot T_1.
$$ 
Consequently, the limit quotient of the system 
of the quotients $\nu^{-1}(V_i) \quot T_1$ 
maps to the limit quotient $(Z \nquotchow T_0) \qquot T_1$.
Since the $\nu^{-1}(V_i) \quot T_1$ form 
a subsystem of the full GIT system of  
$X \nquotchow T_0$, this gives rise to 
a morphism 
$$
(X \nquotchow T_0) \qquot T_1 
\ \to \ 
(Z \nquotchow T_0) \qquot T_1
$$
as needed for the above commutative diagram.
As in the proof of Proposition~\ref{thm:chow2GIT}, 
we may pass to the normalizations and thus obtain a 
morphism
$$
(X \nquotchow T_0) \nquotchow  T_1
\ \to \ 
(Z \nquotchow T_0) \nquotchow  T_1.
$$ 
Now, by the toric case, we have a proper
birational morphism from the toric 
variety on the right hand side onto 
$Z \nquotchow T$.
Using once more Proposition~\ref{thm:chow2GIT},
the assertion follows.
\end{proof}

\section{Toric ambient modifications}
\label{sec:ambmod1}

In this section, we provide a general machinery 
to study the effect of modifications on the Cox 
ring.
Similar to~\cite{HaMMJ}, we use toric embeddings.
In contrast to the geometric criteria given there,
our approach here is purely algebraic, based on 
results of~\cite{Bech}.
The heart is a construction of factorially graded 
rings out of given ones. 

We begin with recalling the necessary algebraic
concepts.
Let $K$ be a finitely generated abelian group 
and $R$ a finitely generated integral $K$-graded 
$\KK$-algebra.
A homogeneous nonzero nonunit $f \in R$ is 
called {\em $K$-prime\/} if $f \mid gh$ with 
homogeneous $g,h \in R$ always implies 
$f \mid g$ or $f \mid h$.
The algebra $R$ is called 
{\em factorially $K$-graded\/} if every
homogeneous nonzero nonunit $f \in R$
is a product of $K$-primes.

We enter the construction of factorially 
graded rings.
Consider a grading of the polynomial ring
$\KK[T_1,\ldots, T_{r_1}]$ by a finitely 
generated abelian group $K_1$ such that 
the variables $T_i$ are homogeneous.
Then we have a pair of exact sequences
$$ 
\xymatrix{
0
\ar[r]
&
{\ZZ^{k_1}}
\ar[r]^{Q_1^*}
&
{\ZZ^{r_1}}
\ar[r]^{P_1}
&
{\ZZ^n}
&
\\
0
\ar@{<-}[r]
&
{K_1}
\ar@{<-}[r]_{Q_1}
&
{\ZZ^{r_1}}
\ar@{<-}[r]_{P_1^*}
&
{\ZZ^n}
\ar@{<-}[r]
&
0
}
$$
where $Q_1 \colon \ZZ^{r_1} \to K_1$ is the degree 
map sending the $i$-th canonical basis vector 
$e_i$ to $\deg(T_i) \in K_1$.
We enlarge $P_1$ to a $n \times r_2$ matrix $P_2$ 
by concatenating further $r_2-r_1$ columns.
This gives a new pair of exact sequences
$$ 
\xymatrix{
0
\ar[r]
&
{\ZZ^{k_2}}
\ar[r]^{Q_2^*}
&
{\ZZ^{r_2}}
\ar[r]^{P_2}
&
{\ZZ^n}
&
\\
0
\ar@{<-}[r]
&
{K_2}
\ar@{<-}[r]_{Q_2}
&
{\ZZ^{r_2}}
\ar@{<-}[r]_{P_2^*}
&
{\ZZ^n}
\ar@{<-}[r]
&
0
}
$$

\begin{construction}
\label{constr:shiftideal}
Given a $K_1$-homogeneous ideal 
$I_1 \subseteq \KK[T_1,\ldots, T_{r_1}]$,
we transfer it to a $K_2$-homogeneous 
ideal 
$I_2 \subseteq \KK[T_1,\ldots, T_{r_2}]$
by taking extensions and contractions
according to the scheme
$$ 
\xymatrix{
{\KK[T_1, \ldots, T_{r_2}]}
\ar[d]_{\imath_2}
& 
&
{\KK[T_1, \ldots, T_{r_1}]}
\ar[d]^{\imath_1}
\\
{\KK[T_1^{\pm 1}, \ldots, T_{r_2}^{\pm 1}]}
\ar@{<-}[r]_{\pi_2^*}
&
{\KK[S_1^{\pm 1}, \ldots, S_n^{\pm 1}]}
\ar[r]_{\pi_1^*}
&
{\KK[T_1^{\pm 1}, \ldots, T_{r_1}^{\pm 1}]}
}
$$
where $\imath_1,\imath_2$ are 
the canonical embeddings and 
$\pi_i^*$ are the homomorphisms 
of group algebras defined by 
$P_i^* \colon \ZZ^n \to \ZZ^{r_i}$.
\end{construction}

Now let $I_1 \subseteq \KK[T_1,\ldots, T_{r_1}]$
be a $K_1$-homogeneous ideal 
and $I_2 \subseteq \KK[T_1,\ldots, T_{r_2}]$ the 
transferred $K_2$-homogeneous ideal.
Our result relates factoriality properties 
of the algebras 
$R_1 := \KK[T_1,\ldots,T_{r_1}] / I_1$
and $R_2 := \KK[T_1,\ldots,T_{r_2}] / I_2$
to each other.

\begin{theorem}
\label{prop:factrings}
Assume $R_1$, $R_2$ are integral,
$T_1,\ldots,T_{r_1}$ define 
$K_1$-primes in $R_1$ and 
$T_{1}, \ldots, T_{r_2}$ 
define $K_2$-primes in $R_2$.
Then the following statements are 
equivalent.
\begin{enumerate}
\item
The algebra $R_1$ is factorially $K_1$-graded.
\item
The algebra $R_2$ is factorially $K_2$-graded.
\end{enumerate}
\end{theorem}

\begin{proof}
First observe that the homomorphisms $\pi_j^*$ embed 
$\KK[S_1^{\pm 1},\ldots, S_{n}^{\pm 1}]$ as the 
degree zero part of the respective $K_j$-grading 
and fit into a commutative diagram 
$$ 
\xymatrix{
I_2
\ar@{}[r]|{\subseteq \qquad\quad}
\ar[d]
&
{\KK[T_1, \ldots, T_{r_2}]}
\ar[d]_{\imath_2}
& 
&
{\KK[T_1, \ldots, T_{r_1}]}
\ar[d]^{\imath_1}
&
I_1
\ar[d]
\ar@{}[l]|{\qquad\quad \supseteq}
\\
I_2'
\ar@{}[r]|{\subseteq \qquad\quad}
&
{\KK[T_1^{\pm 1},\ldots, T_{r_2}^{\pm 1}]}
\ar[rr]^{\psi \colon T_i \mapsto 
         \begin{cases}
\scriptstyle  T_i & \scriptstyle 1 \le i \le r_1 
\\ 
\scriptstyle  1 & \scriptstyle r_1+1 \le i \le r_2
         \end{cases}}
&
&
{\KK[T_1^{\pm 1},\ldots, T_{r_1}^{\pm 1}]}
&
I_1'
\ar@{}[l]|{\qquad\quad \supseteq}
\\
&
I_2''
\ar[u]
\ar[ul]
&
{\KK[S_1^{\pm 1},\ldots, S_{n}^{\pm 1}]}
\ar[ul]^{\pi_2^*}
\ar[ur]_{\pi_1^*}
&
I_1''
\ar[u]
\ar[ur]
&
}
$$

The factor ring $R_1'$ of the extension 
$I_1' := \bangle{\imath_1(I_1)}$ 
is obtained from $R_1$ by localization
with respect to $K_1$-primes $T_1, \ldots, T_{r_1}$:
$$
R_1'
\ := \ 
\KK[T_1^{\pm 1}, \ldots, T_{r_1}^{\pm 1}] / I_1'
\ \cong \  
(R_1)_{T_1 \cdots T_{r_1}}.
$$
The ideal $I_1''$ is the degree zero part of 
$I_1'$.
Thus, its factor algebra is the degree zero 
part of $R_1'$:
$$
R_1''
\ := \ 
\KK[T_1^{\pm 1}, \ldots, T_{r_1}^{\pm 1}]_0 / I_1''
\ \cong \ 
{(R_1')}_0.
$$
Note that $\KK[T_1^{\pm 1}, \ldots, T_{r_1}^{\pm 1}]$ and
hence $R_1'$ admit units in every degree.
Thus, \cite[Thm.~1.1]{Bech} yields that $R_1$ is a 
factorially $K_1$-graded if and only if
$R_1''$ is a UFD.

The homomorphism $\psi$ restricts to an isomorphism 
$\psi_0$ of the respective degree zero parts.
Thus, the shifted ideal $I_2'' := \psi_0^{-1}(I_1'')$
defines an algebra $R_2''$ isomorphic to~$R_1''$:
$$ 
R_2'' 
\ := \ 
\KK[T_1^{\pm 1}, \ldots, T_{r_2}^{\pm 1}]_0 / I_2''
\ \cong \ 
R_1''.
$$
The ideal 
$I_2' := \bangle{\pi_2^*((\pi^*_1)^{-1}(I_1'))}$
has $I_2''$ as its degree zero part and
$\KK[T_1^{\pm 1}, \ldots, T_{r_2}^{\pm 1}]$ admits 
units in every degree.
The associated $K_2$-graded algebra
$$ 
R_2' 
\ := \ 
\KK[T_1^{\pm 1}, \ldots, T_{r_2}^{\pm 1}] / I_2'
$$
is the localization of $R_2$ by the $K_2$-primes
$T_1, \ldots, T_{r_2}$.
Again by~\cite[Thm.~1.1]{Bech} we obtain
that $R_2''$ is a UFD if and only if 
$R_2$ is factorially $K_2$-graded. 
\end{proof}

The following observation is intended for practical 
purposes; it reduces, for example, the number of 
necessary primality tests.

\begin{proposition}
\label{prop:K1primcrit}
Assume that $R_1$ is integral and the canonical map 
$K_2 \to K_1$ admits a section (e.g. $K_1$ is free).
\begin{enumerate}
\item
Let $T_1,\ldots, T_{r_1}$ define $K_1$-primes in $R_1$
and $T_{r_1+1},\ldots, T_{r_2}$ define $K_2$-primes in $R_2$.
If no $T_j$ with $j \ge r_1+1$ divides a $T_i$ with 
$i \le r_1$, then also $T_1,\ldots, T_{r_1}$ define 
$K_2$-primes in $R_2$.
\item
The ring $R_2$ is integral.
Moreover, if $R_1$ is normal and 
$T_{r_1+1},\ldots, T_{r_2}$ 
define primes in $R_2$ (e.g. they are $K_2$-prime 
and $K_2$ is free), then $R_2$ is normal.
\end{enumerate}
\end{proposition}

\begin{proof}
The exact sequences involving the grading groups 
$K_1$ and $K_2$ fit into a commutative diagram
where the upwards sequences are exact and
$\ZZ^{r_2-r_1} \to K_2'$ is an isomorphism:
$$ 
\xymatrix@R15pt{
&
0
&
0
&
0
&
\\
0
\ar@{<-}[r]
&
{K_1}
\ar@{<-}[r]^{Q_1}
\ar[u]
&
{\ZZ^{r_1}}
\ar@{<-}[r]^{P_1^*}
\ar[u]
&
{\ZZ^n}
\ar@{<-}[r]
\ar[u]
&
0
\\
0
\ar@{<-}[r]
&
{K_2}
\ar@{<-}[r]_{Q_2}
\ar[u]
&
{\ZZ^{r_2}}
\ar@{<-}[r]_{P_2^*}
\ar[u]
&
{\ZZ^n}
\ar@{<-}[r]
\ar[u]
&
0
\\
&
{K_2'}
\ar[u]
\ar@{<-}[r]
&
{\ZZ^{r_2-r_1}}
\ar@{<-}[r]
\ar[u]
&
0
\ar[u]
&
\\
&
0
\ar[u]
&
0
\ar[u]
&
&
}
$$
Moreover, denoting by  $K_1' \subseteq K_2$ 
the image of the section $K_1 \to K_2$,
there is a splitting $K_2 = K_2' \oplus K_1'$.
As $K_2' \subseteq K_2$ is the subgroup 
generated by the degrees of $T_{r_1+1}, \ldots, T_{r_2}$,
we obtain a commutative diagram
$$ 
\xymatrix{
{\KK[T_1, \ldots, T_{r_2}]}
\ar[d]_{\imath_2}
& 
&
&
&
\\
{\KK[T_1, \ldots, T_{r_1},T_{r_1+1}^{\pm 1},\ldots, T_{r_2}^{\pm 1}]}
\ar[rrrr]^{\qquad \psi \colon T_i \mapsto 
         \begin{cases}
\scriptstyle  T_i & \scriptstyle 1 \le i \le r_1 
\\ 
\scriptstyle  1 & \scriptstyle r_1+1 \le i \le r_2
         \end{cases}}
&
&
&
&
{\KK[T_1,\ldots, T_{r_1}]}
\\
{\KK[T_1, \ldots, T_{r_1},T_{r_1+1}^{\pm 1},\ldots, T_{r_2}^{\pm 1}]}_{0}
\ar[u]^{\mu}
\ar[urrrr]_{\cong}
&
&
&
&
}
$$
where the map $\mu$ denotes the embedding of the 
degree zero part with respect to the $K_2'$-grading.
By the splitting $K_2 = K_2' \oplus K_1'$, the image 
of $\mu$ is precisely the Veronese subalgebra 
associated to the subgroup $K_1' \subseteq K_2$.
For the factor rings $R_2$ and $R_1$ by the 
ideals $I_2$ and $I_1$, 
the above diagram leads to the following situation
$$ 
\xymatrix{
R_2
\ar[d]_{\imath_2}
& 
&
\\
(R_2)_{T_{r_1+1} \cdots T_{r_2}}
\ar[rr]^{\psi}
&
&
R_1
\\
{\left( (R_2)_{T_{r_1+1} \cdots T_{r_2}} \right)}_{0}
\ar[u]^{\mu}
\ar[urr]_{\cong}
&
&
}
$$

To prove~(i), consider a variable $T_i$ with 
$1 \le i \le r_1$.
We have to show that $T_i$ defines a $K_2$-prime 
element in $R_2$.
By the above diagram,
$T_i$ defines a $K_1'$-prime element in 
$((R_2)_{T_{r_1+1} \cdots T_{r_2}})_{0}$,
the Veronese subalgebra of $R_2$ defined by 
$K_1' \subseteq K_2$.
Since every $K_2$-homogeneous element of
$(R_2)_{T_{r_1+1} \cdots T_{r_2}}$  can be 
shifted by a homogeneous unit into
$((R_2)_{T_{r_1+1} \cdots T_{r_2}})_{0}$,
we see that $T_i$ defines a $K_2$-prime in
$(R_2)_{T_{r_1+1} \cdots T_{r_2}}$.
By assumption, $T_{r_1+1}, \ldots, T_{r_2}$ 
define $K_2$-primes in $R_2$ and are all coprime 
to $T_i$.
It follows that $T_i$ defines a $K_2$-prime in $R_2$.

We turn to assertion~(ii).
As just observed, the degree zero part
$((R_2)_{T_{r_1+1} \cdots T_{r_2}})_{0}$ 
of the $K_2'$-grading is isomorphic to 
$R_1$ and thus integral (normal if $R_1$ is so).
Moreover, the $K_2'$-grading is free in 
the sense that the associated torus 
$\Spec \, \KK[K_2']$ 
acts freely on $\Spec \, (R_2)_{T_{r_1+1} \cdots T_{r_2}}$.
It follows that $(R_2)_{T_{r_1+1} \cdots T_{r_2}}$ 
is integral (normal if $R_1$ is so).
Construction~\ref{constr:shiftideal} gives that
$R_2$ is integral.
Moreover, if $T_{r_1+1}, \ldots, T_{r_2}$ define 
primes in $R_2$, we can conclude that 
$R_2$ is normal.
\end{proof}

Let us apply the results to Cox rings. 
We first briefly recall 
the basic definitions and facts; for details 
we refer to~\cite{ADHL}.
For a normal variety $X$ with 
finitely generated divisor class group
$\Cl(X)$ and $\Gamma(X,\mathcal{O}^*) = \KK^*$, 
one defines its Cox ring as the graded ring
\begin{eqnarray*}
\mathcal{R}(X) 
& := &
\bigoplus_{\Cl(X)} \Gamma(X,\mathcal{O}(D)).
\end{eqnarray*}
This ring is factorially $\Cl(X)$-graded.
Moreover, if $\mathcal{R}(X)$ is finitely 
generated, then
one can reconstruct $X$ from $\mathcal{R}(X)$
as a good quotient of an open subset of 
$\Spec \, \mathcal{R}(X)$ by the action of 
$\Spec \, \KK[\Cl(X)]$.

Now return to the setting fixed at the beginning 
of the section and assume in addition that 
the columns of $P_2$ are pairwise different 
primitive vectors in $\ZZ^n$ and those of $P_1$ 
generate $\QQ^{n}$ as a convex cone.
Suppose we have toric Cox constructions
$\pi_i \colon \hat{Z}_i \to Z_i$, 
where $\hat{Z}_i \subseteq \KK^{r_i}$ are 
open toric subvarieties and $\pi_i$ 
are toric morphisms defined by $P_i$,
see~\cite{Cox}.
Then the canonical map $Z_2 \to Z_1$ 
is a toric modification.
Consider the ideal $I_1$ as discussed 
before and the geometric data
$$
\bar{X}_1 \ := \ V(I_1) \subseteq \KK^{r_1},
\qquad
\hat{X}_1 \ := \ \bar{X}_1 \cap \hat{Z}_1,
\qquad
X_1 \ := \ \pi_1(\hat{X}_1) \subseteq Z_1.
$$
Assume that $R_1$ is factorially $K_1$-graded
and $T_1, \ldots, T_{r_1}$ define pairwise
nonassociated prime elements in $R_1$.
Then $R_1$ is the Cox ring of $X_1$,
see~\cite{ADHL}.
Our statement concerns the Cox ring of the 
proper transform 
$X_2 \subseteq Z_2$ of $X_1 \subseteq Z_1$ 
with respect to $Z_2 \to Z_1$.

\begin{corollary}
\label{cor:ambientblowup}
In the above setting, assume that $R_2$ is 
normal and the variables $T_1, \ldots, T_{r_2}$ 
define pairwise nonassociated $K_2$-prime 
elements in $R_2$. Then the $K_2$-graded 
ring $R_2$ is the Cox ring of $X_2$.
\end{corollary}

\begin{proof}
According to Theorem~\ref{prop:factrings}, 
the ring $R_2$ is factorially $K_2$-graded.
Moreover, with the toric Cox construction 
$\pi_2 \colon \rq{Z}_2 \to Z_2$, we obtain
that $R_2$ is the algebra of functions 
of the closure $\rq{X}_2 \subseteq \rq{Z}_2$
of $\pi_2^{-1}(X_2 \cap \TT^{r_2})$.
Thus, \cite{ADHL} yields that $R_2$ is the
Cox ring of $X_2$.
\end{proof}

\begin{example}
\label{ex:kubik}
We start with the UFD $R_1 = \KK[T_1,\ldots,T_8] / I_1$,
where the ideal $I_1$ is defined as
$$
I_1 
\ = \ 
\bangle{T_1T_2 + T_3T_4 + T_5T_6 + T_7T_8}.
$$
The ideal $I_1$ is homogeneous with respect 
to the standard grading given by
$Q_1 = [1,\ldots,1]$.
A Gale dual is $P_1 = [e_0,e_1, \ldots,e_7]$, 
where $e_0 = -e_1 -\ldots - e_7$ and $e_i$ are 
the canonical basis vectors.
Concatenating $e_1+e_3$ gives a matrix $P_2$. 
The resulting UFD is $R_2 = \KK[T_1,\ldots,T_9] / I_2$ 
with
$$
I_2
\ = \ 
\bangle{T_1T_2T_9 + T_3T_4T_9 + T_5T_6 + T_7T_8}.
$$ 
\end{example}

Corollary~\ref{cor:ambientblowup}  in fact 
characterizes the modifications preserving 
finite generation of Cox rings.
Together with the combinatorial contraction 
criterion~\cite[Prop~6.7]{HaMMJ}, it gives 
the following.

\goodbreak

\begin{remark}
Let $X' \to X$ be a birational morphism 
of normal projective varieties, 
where $X$ is a Mori dream space, i.e.~has finitely 
generated Cox ring. Then the following statements are 
equivalent.
\begin{enumerate}
\item
$X'$ is a Mori dream space.
\item
$X' \to X$ arises from a toric ambient
modification as in Corollary~\ref{cor:ambientblowup}.
\end{enumerate}
\end{remark}

\section{Proof of Theorem~\ref{mainthm2}}
\label{sec:quadrics1}

We approach the Chow quotient via toric embedding.
The idea then is to obtain the Cox ring via toric 
ambient modifications. An essential step for this 
is the explicit description of the rays of certain
Gelfand-Kapranov-Zelevinsky decompositions given
in Proposition~\ref{cor:gkzrays}; note that in the 
setting of polytopes related statements implicitly 
occur in literature, e.g.~\cite{HeJo,He}.

Recall that the Gelfand-Kapranov-Zelevinsky 
decomposition 
associated to a matrix $P \in \Mat(n,r+1;\ZZ)$;
is the fan $\Sigma$  in $\QQ^n$ with the cones
$\sigma(v) = \cap_{v \in \tau^\circ} \tau$, where 
$v \in \QQ^n$ and $\tau$ runs through the $P$-cones, 
i.e., the cones generated by some of the columns 
$p_0,\ldots, p_r$ of~$P$.
Fix a Gale dual matrix $Q \in \Mat(k,r+1;\ZZ)$, where 
$r+1 = k+n$, and denote the columns of $Q$ by 
$q_0,\ldots,q_r$.
Then we have mutually dual exact sequences of 
rational vector spaces
$$ 
\xymatrix{
0
\ar[r]
&
\QQ^k
\ar[r]^{Q^*}
&
\QQ^{r+1}
\ar[r]^{P}
&
\QQ^{n}
\ar[r]
&
0
\\
0
\ar@{<-}[r]
&
\QQ^k
\ar@{<-}[r]_{Q}
&
\QQ^{r+1}
\ar@{<-}[r]_{P^*}
&
\QQ^n
\ar@{<-}[r]
&
0
}
$$
By a {\em $Q$-hyperplane\/} we mean a linear hyperplane 
in $\QQ^k$ generated by some of the columns $q_0, \ldots, q_{r}$.
Given a $Q$-hyperplane we write it as the kernel $u^\perp$ of 
a linear form $u$ and associate to it a ray in $\QQ^{n}$ 
as follows:
\begin{eqnarray*}
\varrho(u)
& := & 
\mathrm{cone}\left(\sum_{u(q_i)> 0} u(q_i)p_i\right).
\end{eqnarray*}
It turns out that $\varrho(u) = \varrho(-u)$ holds 
and thus the ray is well defined.
We say that a column $q_i$ of $Q$ is {\em extremal\/} if 
it does not belong to the relative interior of the
``movable cone'' $\cap_i \, \cone(q_j; \; j \ne i)$.

\begin{proposition}
\label{cor:gkzrays}
Let $Q$ and $P$ be Gale dual matrices as before,
assume that the columns of $P$ are pairwise linearly 
independent nonzero vectors generating $\QQ^{n}$ 
as a cone and let $\Sigma$ be the Gelfand-Kapranov-Zelevinsky 
decomposition  associated to $P$.
\begin{enumerate} 
\item 
If a ray $\varrho \in \Sigma$ is the intersection of two 
$P$-cones, then $\varrho = \varrho(u)$ holds with a 
$Q$-hyperplane $u^\perp$.
\item
If $k=2$ holds, then every ray of $\Sigma$ 
can be obtained as an intersection of two $P$-cones.
\item 
Assume  $k=2$ and fix nonzero linear forms $u_i$ 
with $u_i \perp q_i$. 
Then the rays of $\Sigma$ are 
$\cone(p_0), \ldots, \cone(p_r)$
and the $\varrho(u_i)$ with $q_i$ not extremal.
\end{enumerate}
\end{proposition}

The proof relies on the fact that $\Sigma$ 
describes the lifts of regular $Q$-subdivisions.
We adapt the precise formulation of this statement 
to our needs.
Let $\gamma \subseteq \QQ^{r+1}$ be the positive orthant 
and define a {\em $\gamma$-collection\/} 
to be a set $\mathfrak{B}$ of faces of $\gamma$ such that 
any two $\gamma_1,\gamma_2 \in \mathfrak{B}$ admit 
an {\em invariant separating linear form\/} 
$f$ in the sense that
$$ 
P^*(\QQ^n) \subseteq f^\perp,
\qquad
f_{\vert \gamma_1}  \ge 0,
\qquad
f_{\vert \gamma_2}  \le 0,
\qquad
f^\perp \cap \gamma_i = \gamma_1 \cap \gamma_2.
$$
Write $\mathfrak{B}_1 \le \mathfrak{B}_2$ if for 
every $\gamma_1 \in \mathfrak{B}_1$ there is 
a $\gamma_2 \in \mathfrak{B}_2$ with 
$\gamma_1 \subseteq \gamma_2$.
Moreover, call a $\gamma$-collection $\mathfrak{B}$ 
{\em normal\/} if it cannot be enlarged as a 
$\gamma$-collection and the images $Q(\gamma_0)$, 
where $\gamma_0 \in \mathfrak{B}$, form the normal fan 
of a polyhedron. 
For a face $\gamma_0 \preceq \gamma$, we denote
by $\gamma_0^* = \gamma_0^\perp \cap \gamma^\vee$ 
the corresponding face of the dual cone $\gamma^\vee$.

Now assume that the columns of $P$ are pairwise different
nonzero vectors.
Then~\cite[Sec.~II.2]{ADHL} provides us 
with an order-reversing bijection
$$
\left\{ 
\text{normal } \gamma \text{-collections}
\right\}
\ \to \
\Sigma,
\qquad\qquad
\mathfrak{B}
\ \mapsto \
\bigcap_{\gamma_0 \in \mathfrak{B}} P(\gamma_0^*).
$$

\begin{proof}[Proof of Proposition~\ref{cor:gkzrays}]
We prove (i).
Let $\varrho = P(\gamma_1^*)\cap P(\gamma_2^*)$ 
with $\gamma_1,\gamma_2 \preceq \gamma$.
We may assume that the relative interiors
$P(\gamma_1^*)^\circ$ and $P(\gamma_2^*)^\circ$ 
intersect nontrivially.
Then $\gamma_1$ and $\gamma_2$ admit an invariant 
separating linear form 
$f = Q^*(u)$ with a linear form $u$ on $\QQ^k$.
In terms of the components of 
$f_i = u(q_i)$ of $f$, we have
$$
\gamma_1 \ = \ \cone(e_i; \; f_i \ge 0),
\qquad\qquad
\gamma_2 \ = \ \cone(e_i; \; f_i \le 0).
$$
Write $f = f^+ - f^-$ with the unique vectors 
$f^+, f^- \in \QQ^{r+1}$ 
having only non-negative components.
Then $P(f) = 0$ gives $P(f^+) = P(f^-)$.
We conclude $\varrho = \cone(P(f^+))$ and 
the assertion follows.

We prove (ii) and (iii).
The rays of $\Sigma$ arise from normal 
$\gamma$-collections which are submaximal
with respect to ``$\le$'' in the sense that 
the only dominating $\gamma$-collection 
is the trivial collection 
$\bangle{\gamma}$ consisting of all faces 
$\gamma_0\preceq\gamma$ which are
invariantly separable from $\gamma$.
There are precisely two types of such submaximal 
collections:
\begin{itemize}
\item
the normal $\gamma$-collections 
$\mathfrak{B} = \bangle{\gamma_0}$, 
where $\gamma_0 \precneqq \gamma$ 
is a facet satisfying $Q(\gamma_0) = Q(\gamma)$,
\item
the normal $\gamma$-collections 
$\mathfrak{B} = \bangle{\gamma_1,\gamma_2}$, 
where $\gamma_1,\gamma_2 \precneqq \gamma$ are 
invariantly separable from each other and 
satisfy
$$
\gamma_i \ = \ Q^{-1}(Q(\gamma_i)) \cap \gamma,
\qquad
Q(\gamma) \ = \ Q(\gamma_1) \cup Q(\gamma_2).
$$
\end{itemize}
The submaximal $\gamma$-collections of the first type 
give the rays $\cone(p_i) \in \Sigma$ with $q_i$
not extremal.
If $q_i$ is extremal, then the (unique) $\gamma$-collection 
of the second type with 
$Q(\gamma_1) = \cone(q_j; \; j \ne i)$ defines 
the ray $\cone(p_i)$.
The remaining rays of $\Sigma$ are of the 
form $\varrho = P(\gamma_1^*) \cap P(\gamma_2)^*$ 
with the remaining collections of the second type.
\end{proof}

\begin{remark}
Statements~(ii) and~(iii) of Proposition \ref{cor:gkzrays} 
hold as well for pairs $P,Q$, where the columns of $Q$ 
generate the cone over a so called {\em totally-2-splittable\/} 
polytope; these have been studied in~\cite{HeJo,He}.
\end{remark}

As a further preparation of the proof of 
Theorem~\ref{mainthm2} we have to 
specialize the discussion of Section~\ref{sec:ambmod1}
to the case of a single defining equation.
The following notion will be used for an explicit 
description of the transferred ideal.

\goodbreak

\begin{definition}
\label{def:weaklift}
Consider an $n \times (r+1)$ matrix $P$
and an $n \times l$ matrix $B$, both integral.
A {\em weak $B$-lifting (with respect to $P$)\/} 
is an integral $(r+1) \times l$ matrix $A$ allowing a  
commutative diagram
$$ 
\xymatrix{
&
{\ZZ^{r+1+l}}
\ar[dl]_{\genfrac{}{}{0pt}{}{e_i \mapsto e_i}{e_j \mapsto m_je_j}}
\ar[dr]^{[E_{r+1},A]}
&
\\
{\ZZ^{r+1+l}}
\ar[d]_{[P,B]}
&&
{\ZZ^{r+1}}
\ar[d]^{P}
\\
{\ZZ^{n}}
\ar[rr]_{E_n}
&&
{\ZZ^{n}}
}
$$
where the $e_i$ are the first $r+1$, the $e_j$ the last $l$
canonical basis vectors of $\ZZ^{r+1+l}$,
the $m_{j}$ are positive 
integers and $E_n, E_{r+1}$ denote the unit matrices
of size $n,r+1$ respectively.
\end{definition}

Note that weak $B$-liftings $A$ always exist.
Given such $A$, consider the following 
homomorphism of Laurent polynomial rings:
\begin{align*}
\psi_{A} \colon
\KK[T_0^{\pm 1}, \ldots, T_{r}^{\pm 1}]
\ &\to \ 
\KK[T_0^{\pm 1}, \ldots, T_{r}^{\pm 1},S_1^{\pm 1},\ldots, S_l^{\pm 1}],
\\
\sum \alpha_\nu T^\nu 
\ &\mapsto \ 
\sum \alpha_\nu T^\nu S^{A^t \cdot \nu}. 
\end{align*}
Set $K_1 := \ZZ^{r+1} / P^*(\ZZ^n)$. Then the 
left hand side algebra is $K_1$-graded 
by assigning to the $i$-th 
variable the class of $e_i$ in $K_1$.

\begin{lemma}
\label{lem:befreiung}
In the above notation, let
$g_1 \in \KK[T_0^{\pm 1}, \ldots, T_{r}^{\pm 1}]$ 
be a $K_1$-homogeneous polynomial.
\begin{enumerate}
\item
We have $T^\nu S^\mu  \psi_A(g_1) = g_2'$ with
 $\nu \in \ZZ^{r+1}$, $\mu \in \ZZ^l$ and 
a unique monomial free
$g_2' \in \KK[T_0, \ldots, T_{r},S_1,\ldots, S_l]$.
\item
The polynomial $g_2'$ is of the form 
$g_2' = g_2(T_0,\ldots,T_{r+1},S_1^{m_1},\ldots,S_l^{m_1})$
with a 
$g_2 \in \KK[T_0, \ldots, T_{r},S_1,\ldots, S_l]$
not depending on the choice of $A$.
\item 
If, in the setting of Construction~\ref{constr:shiftideal},
we have $I_1 = \bangle{g_1}$, then the transferred ideal 
is given by $I_2 = \bangle{g_2}$.
\item
The variable $T_i$ defines a prime element in 
$\KK[T_0, \ldots, T_{r+l+1}] / \bangle{g_2}$ 
if and only if the polynomial 
$g_2(T_1, \ldots, T_{i-1},0,T_{i+1}, \ldots, T_{r + l +1})$ 
is irreducible.
\end{enumerate}
\end{lemma}

\begin{proof}
Consider the commutative diagram of group algebras 
corresponding to the dualized diagram~\ref{def:weaklift}.
There, $\psi_A$ occurs as the homomorphism of group 
algebras defined by the transpose $[E_{r+1},A]^*$.
Let $T^\kappa$ be any monomial of $g_1$.
Then $g_1' := T^{-\kappa}g_1$ gives rise to the 
same $g_2$, but $g_1'$ is of $K_1$-degree zero and
hence a pullback $g_1' = \psi_{P^*}(h)$.
The latter allows to use commutativity of the diagram
which gives~(i) and (ii). Assertions~(iii) and~(iv)
are clear.
\end{proof}

\begin{proof}[Proof of Theorem~\ref{mainthm2}]
Recall that we consider the quadric 
$X = V(g_1) \subseteq \PP_r$ with 
$g_1 = T_0T_1 + \ldots + T_{r-1}T_r$, 
where we replace the last term with 
$T_r^2$ in the case of an even $r$,
and a $\KK^*$-action on $\PP_r$, 
given by weights $\zeta_0, \ldots, \zeta_r$
such that $g_0$ is of degree zero and, 
in particular, $X$ is invariant. 

In a first step, we construct a suitable 
GIT quotient $X_1$ of the $\KK^*$-action on~$X$.
Lifting the above data to $\KK^{r+1}$ gives
$\bar{X} := V(g_1) \subseteq \KK^{r+1}$ 
which is invariant under the action of 
$\TT^2 = \KK^* \times \KK^*$ on $\KK^{r+1}$
given by the weight matrix
\begin{eqnarray*}
Q 
& := & 
\left[
\begin{array}{ccc}
\zeta_0 & \ldots & \zeta_r
\\
1 & \ldots & 1
\end{array}  
\right]
\end{eqnarray*}
Consider the weight $w = (0,1)$ of $\TT^2$ and the 
associated set of semistable points 
$\hat{Z}_1 \subseteq \KK^{r+1}$, that means the union 
of all localizations $\KK^{r+1}_f$, where $f$ is 
homogeneous with respect to some positive multiple 
of $w$. Then $\hat{Z}_1$ is 
a toric open subset, and with 
$\hat{X}_1 := \bar{X} \cap \hat{Z}_1$
we obtain a commutative diagram
$$ 
\xymatrix{
{\hat{X}_1}
\ar@{}[r]|\subseteq
\ar[d]_{\quot \TT^2}
&
{\hat{Z}_1}
\ar[d]^{\quot \TT^2}
\\
X_1
\ar[r]
&
Z_1
}
$$
where the induced map $X_1 \to Z_1$ of quotients 
is a closed embedding.
We are in the setting presented before
Corollary~\ref{cor:ambientblowup}.
In particular, $\hat{Z}_1 \to Z_1$ is a toric Cox 
construction with a Gale dual $P$ of $Q$
as describing matrix;
note that the columns of $P$ generate $\ZZ^{r-1}$ as 
a lattice.
Moreover, the Cox ring of $X_1$ is the $\ZZ^2$-graded 
ring
\begin{eqnarray*}
R_1 
& = & 
\KK[T_0, \ldots, T_r] \ / \ \bangle{g_1}.
\end{eqnarray*}
Observe that $X_1$ is as well the $\KK^*$-quotient 
of the image of $\hat{X}_1$ in $X$ which in turn
is the set of semistable points of a suitable 
linearization of $\mathcal{O}(1)$.

Set $n := r-1$ and consider the Gelfand-Kapranov-Zelevinsky 
decomposition~$\Sigma$ associated to $P$.
Then, according to Proposition~\ref{prop:chow2GIT}, the toric 
variety $Z_2$ determined by $\Sigma$ is the normalized 
Chow quotient of the $\KK^*$-action on $\PP_{r}$. 
Moreover, let $X_2 \subseteq Z_2$ denote the proper transform 
of $X_1 \subseteq Z_1$ under the toric morphism $Z_2 \to Z_1$.
Then Proposition~\ref{thm:chow2GIT} tells us that 
$X_2$ and the Chow quotient $X \quotchow \KK^*$ share the 
same normalization. 

We will now show that $X_2$ is in fact normal and that 
its Cox ring is as claimed in the Theorem.
As before, put the primitive generators $b_1, \ldots, b_l$ 
of rays of $\Sigma$ differing from columns of $P$ 
into a matrix $B$ and choose a weak $B$-lifting $A$
with respect to~$P$;
using the fact that the columns of $P$ generate $\ZZ^n$,
we can choose the numbers $m_j$ all equal to one.
With the shifted row sums 
$\eta_0, \eta_2, \ldots, \eta_{r-1}$
we set
\begin{eqnarray*}
g_2
& := &
\begin{cases} 
T_0T_1S^{\eta_0} 
+
T_2T_3S^{\eta_2} 
+ \ldots + 
T_{r-1}T_r S^{\eta_{r-1}},
&
r \text{ odd},
\\
T_0T_1S^{\eta_0} 
+  \ldots + 
T_{r-2}T_{r-1}S^{\eta_{r-2}} 
+
T_r^2 S^{\eta_{r}},
&
r \text{ even.}
\end{cases}
\end{eqnarray*}
Lemma~\ref{lem:befreiung} then ensures that 
$I_2 := \bangle{g_2}$ is 
the transferred ideal of $I_1 := \bangle{g_1}$ in the 
sense of Construction~\ref{constr:shiftideal}; define 
$P_1 := P$ and $P_2 := [P,B]$ to adapt the settings.
Consider the ring 
\begin{eqnarray*}
R_2
& = & 
\KK[T_0,\ldots,T_r,S_1,\ldots, S_l] \ / \ \bangle{g_2}.
\end{eqnarray*}
Our task is to show that the variables $S_1, \ldots, S_l$ 
define prime elements in $R_2$. Then 
Proposition~\ref{prop:K1primcrit} tells us that $R_2$ and 
thus $X_2$ are normal and Corollary~\ref{cor:ambientblowup}
yields that the Cox ring of $X_2$ is $R_2$ together with 
the $\ZZ^{2+l}$-grading defined by a Gale dual $Q_2$ of 
$P_2 = [P,B]$.

Suitably renumbering the variables $T_i$, we achieve 
that  $|\zeta_{r-3}|,\ldots,|\zeta_r|$ 
are minimal among all $|\zeta_i|$ in the case of odd $r$
and, similarly, in the case of even $r$, we have 
$\zeta_{r-3} = \zeta_{r-2} = \zeta_{r-1} = 0$.
In order to see that the $S_j$ define primes,
it suffices to show that, according to odd and even $r$,
$$
g_2
\ = \ 
T_{r-3}T_{r-2}+T_{r-1}T_r+h,
\quad
\text{ or } 
\quad
g_2
\ = \ 
T_{r-2}T_{r-1}+T_{r}^2+h,
$$
holds with a polynomial $h \in \KK[T_0,\ldots,T_r,S_1,\ldots, S_l]$ 
not depending on the last four (three) $T_i$,
see Lemma~\ref{lem:befreiung}~(iv).
This in turn is seen by constructing a suitable 
weak $B$-lifting via the description of the rays 
through $b_1,\ldots, b_l$ provided by 
Proposition~\ref{cor:gkzrays}.
Each $b_j$ (or a suitable integral multiple) stems 
from a $Q$-hyperplane and the $u_j$ can be chosen 
to be nonpositive on the last four (three) $q_i$. 
Putting $\max(0,u_j(q_i))$ into 
a matrix $A'$ gives a weak $B$-lifting $A'$ with 
$A'_{i*}=0$ for the last four (three) rows. 
By Lemma~\ref{lem:befreiung}, the weak
$B$-lifting $A'$ yields the same $g_2$
which now has the desired form.
\end{proof}

\begin{example}
Consider the quadric 
$X = V(T_0T_1+ T_2T_3 + T_4T_5+ T_6^2)\subseteq \PP_6$ 
and the action of $\KK^*$ on $\PP_6$ given by
\begin{eqnarray*}
 t\cdot [x_0,\ldots,x_6]
& := &
[t^{-2}x_0,t^{2}x_1,t^{-1}x_2,t^{1}x_3,x_4,x_5,x_6].
\end{eqnarray*}
An integral Gale dual $P$ of the extended weight 
matrix $Q$ is of size $5 \times 7$ and explicitly 
given as
$$
\left[
\begin{array}{rrrrrrr}
-1 & -1 & 1 & 1 & 0 & 0 & 0\\
0 & 0 & 0 & 0 & -1 & 1 & 0\\
0 & -1 & -1 & 1 & 0 & 1 & 0\\
0 & 0 & 1 & 1 & -1 & -1 & 0\\
0 & 0 & 0 & 0 & -1 & 0 & 1
\end{array}
\right].
$$
Computing the associated 
Gelfand-Kapranov-Zelevinsky decomposition 
we see that it comes with one new ray,
namely
$$ 
b_1 
\ = \
(-1,0,-1,1,0)
\ = \
2p_0 + p_2,
$$
where $p_0, \ldots, p_6$ are the columns of $P$.
The Cox ring of the normalized Chow quotient 
$X \nquotchow \KK^*$ is the ring
\begin{eqnarray*}
\mathcal{R}(X \nquotchow \KK^*)
&  = & 
\KK[T_0,\ldots,T_6,S_1]
\ / \ 
\bangle{T_0T_1S_1^2  + T_2T_3S_1  + T_4T_5  + T_6^2}
\end{eqnarray*}
together with the grading by $\Cl(X) = \ZZ^3$
via a Gale dual of $[p_0,\ldots, p_6,b_1]$,
i.e. the  degrees of the variable are the 
columns of
$$
\left[
\begin{array}{rrrrrrrr}
-2 & 2 & -1 & 1 & 0 & 0 & 0 & 0
\\
1 & 1 & 1 & 1 & 1 & 1 & 1 & 0
\\
2 & 0 & 1 & 0 & 0 & 0 & 0 & -1
\end{array} 
\right].
$$
\end{example}

\begin{remark}
The setting of Theorem~\ref{mainthm2} can also be 
interpreted in terms of Mori theory.
There are (up to isomorphism) finitely many
normal projective varieties $Y_1, \ldots, Y_s$ 
sharing as their Cox ring a given 
$R_1 = \KK[T_0,\ldots,T_n] / \bangle{g_1}$ 
with its $\ZZ^2$-grading coming from
the extended weight matrix~$Q$. 
Each $Y_i$ is a GIT-quotient of the induced 
$\KK^*$-action on the quadric 
$X = V(g_1) \subseteq \PP_r$ 
and thus dominated in universal manner by 
the normalized Chow quotient $Y = X \nquotchow \KK^*$.
Thus, $Y$ is the ``Mori master space'' controlling
the whole class of small birational relatives $Y_i$. 
This picture obviously extends to all Mori dream 
spaces, and it is a natural desire to study the
geometry of the Mori master spaces.
\end{remark}

\section{Proof of Theorem~\ref{mainthm1}}
\label{sec:quadrics2}

The main idea of the proof is to consider instead 
of the Chow quotient its ``weak tropical resolution''
and to use intrinsic symmetry of the latter space.
This approach applies also to problems beyond 
$\KK^*$-actions on quadrics; we therefore develop 
it in sufficient generality.
We begin with recalling the necessary concepts from 
tropical geometry.

Let $f$ be a Laurent polynomial in $n$ variables.
The Newton polytope $B_f \subseteq \QQ^n$ 
is the convex hull over the exponent vectors of $f$. 
The tropical variety $\trop(V(f))$
of the zero set $V(f) \subseteq \TT^n$ lives
in $\QQ^n$ and is defined to be the union 
of all $(n-1)$-dimensional cones of the normal fan 
of~$B_f$.
The tropical variety of an arbitrary closed subset 
$Y \subseteq \TT_n$ is the intersection
$\trop(Y)$ over all $\trop(V(f))$, where 
$f$ runs through the ideal of $Y$.
It turns out that $\trop(Y)$ is the support 
of an (in general not unique and not pointed) 
fan in $\QQ^n$.

\begin{definition}
Consider a toric variety $Z$ defined by a fan
$\Sigma$ in $\QQ^n$ and an irreducible
subvariety $Y \subseteq Z$ intersecting the 
big torus $\TT^n \subseteq Z$ nontrivially.
We call the embedding $Y \subseteq Z$ 
{\em weakly tropical\/} if the support 
$\vert \Sigma \vert \subseteq \QQ^n$ 
equals the tropical variety 
$\trop(Y\cap \TT^n) \subseteq \QQ^n$. 
\end{definition}

\begin{remark}
Any tropical embedding in the sense of Tevelev~\cite{Tev}
is weakly tropical.
If $Y \subseteq Z$ is a weakly tropical subvariety 
of a toric variety $Z$,
then, by~\cite[Sec.~14]{Gu}, 
for any toric orbit $\TT^n \mal z \subseteq Z$
intersecting $Y$ nontrivially, we have 
\begin{eqnarray*}
\dim(Z) \  - \ \dim(\TT^ n \mal z)
& = & 
\dim(Y) \ - \ \dim(\TT^ n \mal z \cap Y).
\end{eqnarray*}
\end{remark}

\begin{construction}[Weak tropical resolution]
\label{constr:tropresgen}
Let $Z$ be a complete toric variety arising from 
a fan $\Sigma$ in $\QQ^n$ and
$Y \subseteq Z$ an irreducible subvariety
intersecting the big torus $\TT^n \subseteq Z$
nontrivially.
Fix a fan structure $\Sigma_Y$ carried on the 
tropical variety $\trop(Y \cap \TT^n) \subseteq \QQ^n$ 
for $Y \cap \TT^n$ and consider the coarsest 
common refinement
$$
\Sigma' 
\quad := \quad
\Sigma
\ \sqcap \
\Sigma_Y
\quad = \quad
\{\tau \cap \sigma; \; 
\sigma \in \Sigma, \
\tau \in \Sigma_Y\}
$$
of the fans $\Sigma$ and $\Sigma_Y$.
Then the canonical map of fans $\Sigma' \to \Sigma$ 
defines a birational toric morphism $Z' \to Z$ 
of the associated toric varieties.
With the proper transform $Y' \subseteq Z'$ of 
$Y \subseteq Z$, we obtain a proper birational map 
$Y' \to Y$ which we call a {\em weak tropical 
resolution\/} of $Y \subseteq Z$. 
\end{construction}

\begin{proof}
The only thing to show is properness of the 
morphism $Y' \to Y$.
But this follows directly from Tevelev's 
criterion~\cite[Prop.~2.3]{Tev}.
\end{proof}

The use of passing to the weak tropical resolution 
in our context is that it enables us to divide out 
torus symmetries in a controlled manner.
This leads to an explicit version of~\cite[Thm.~1.2]{HaSu}
relating the Mori dream space property of a variety 
to the Mori dream space property of a certain quotient.

\begin{construction}
\label{constr:nullquot}
Consider a toric variety $Z$ arising 
from a fan $\Sigma$ in $\QQ^r$,
and a weakly tropical embedded 
subvariety $Y \subseteq Z$.
Suppose that  $Y$ is invariant under 
the action of a subtorus $T \subseteq \TT^r$.
Set
$$ 
Z_0 
\ := \ 
\{
z \in Z; \ 
\dim (\TT^r \mal z) \ge r-1, \ 
T_z \text{ finite}
\},
\qquad 
\qquad 
Y_0 
\ := \ 
Y \cap Z_0.
$$
Then $Z_0 \subseteq Z$ is an open toric subset
corresponding to a subfan $\Sigma_0 \preceq \Sigma$ 
with certain rays $\varrho_1, \ldots, \varrho_s$ of 
$\Sigma$ as its maximal cones.
Let the matrix $P \in \Mat(n,r;\ZZ)$ describe an 
epimorphism $\pi \colon \TT^r \to \TT^n$ with 
$\ker(\pi) = T$ and consider the following fan in $\ZZ^n$:
$$ 
\Delta_0 
\ := \ 
\{0,P(\varrho_1), \ldots, P(\varrho_s)\}.
$$
Note that $\varrho_1, \ldots, \varrho_s$
are precisely the rays of $\Sigma$ which are not 
contained in $\ker(P)$.
The matrix $P$ determines a toric morphism 
$Z_0 \to Z \quotnull T$ onto the toric variety 
associated to $\Delta_0$.
We define $Y \quotnull T \subseteq Z \quotnull T$ 
to be the closure of the image $\pi(Y \cap \TT^r)$.
\end{construction}

\begin{remark}
The tropical variety 
$\trop(Y \quotnull T \cap \TT^n)$ contains all rays 
$P(\varrho_1), \ldots, P(\varrho_s)$ of 
the fan $\Delta_0$.
If there is a fan $\Delta$ in $\ZZ^n$ having  
$\trop(Y \quotnull T \cap \TT^n)$ as its support
and $P(\varrho_1), \ldots, P(\varrho_s)$ as its rays, 
then $Y \quotnull T$ admits a weakly tropical 
completion with boundary of codimension at least two.
\end{remark}

\begin{proposition}
\label{prop:tropembmdschar}
Consider a toric variety $Z$, a 
weakly tropical subvariety 
$Y \subseteq Z$ and suppose 
that $Y$ is invariant under 
the action of a subtorus 
$T \subseteq \TT^r$. 
Then the following statements are 
equivalent.
\begin{enumerate}
\item
The normalization of $Y$ has finitely generated Cox ring.
\item
The normalization of $Y \quotnull T$ has finitely generated 
Cox ring.
\end{enumerate}
\end{proposition}

\begin{proof}
Let $\nu\colon\tilde Y\to Y$ be the normalization map. 
By $W \subseteq Y$ we denote the open $T$-invariant 
subset consisting of all points $y \in Y$ having 
a finite isotropy group $T_y$.
The fact that $Y \subseteq Z$ is tropically embedded
ensures that $Y_0 \subseteq W$ has a complement of 
codimension at least two in $W$. 
This property is preserved when passing to the 
respective normalizations $\tilde W:=\nu^{-1}(W)$ 
and $\tilde Y_0:=\nu^{-1}(Y_0)$. 
In particular the separations in the sense 
of~\cite[p.~978]{HaSu} of the corresponding quotients 
$\tilde W/T$ and $\tilde Y_0/T$ have the same 
Cox rings.
Since normalizing commutes with taking quotients and 
separating, the latter space is isomorphic to the 
normalization of ${Y\quotnull T}$.
Thus the assertion follows from~\cite[Theorem 1.2]{HaSu}.
\end{proof}

\begin{proposition}
\label{prop:mdscrit}
Let $Z$ be a toric variety, $Y \subseteq Z$ 
a complete subvariety which is invariant 
under a subtorus $T$ of the big torus of $Z$
and  $Y' \to Y$ be a weak tropical resolution.
If the normalization of $Y' \quotnull T$ has 
finitely generated Cox ring, then the
normalization $\tilde Y$ of $Y$ is a Mori 
dream space.
\end{proposition}

\begin{proof}
Since the normalization of $Y' \quotnull T$ has finitely 
generated Cox ring, Proposition~\ref{prop:tropembmdschar} 
shows that the normalization $Y''$ of $Y'$ has finitely 
generated Cox ring and thus is a Mori dream space. 
The canonical morphism $\pi \colon Y'' \to \tilde Y$ 
is proper and birational.
In order to see that $\tilde Y$ is a Mori dream space,
we may apply the general~\cite[Thm.~10.4]{Ok},
or look at a suitable sheaf 
$\mathcal{S} = \oplus_K \mathcal{O}_{\tilde Y}(D)$ 
of divisorial algebras on $\tilde Y$ mapping onto 
the Cox sheaf $\mathcal{R}$ of $\tilde Y$. 
By properness of $\pi$, we obtain 
$\mathcal{S} = \pi_* \mathcal{S}''$ over 
the set $W \subseteq \tilde Y$ of regular points for
$\mathcal{S}'' = \oplus_{K} \mathcal{O}_{Y''}(\pi^*(D))$.
Since $Y''$ is a Mori dream space, 
$\Gamma(\pi^{-1}(W),\mathcal{S}'')$
is finitely generated.
This implies finite generation of the Cox ring
$\mathcal{R}(\tilde Y) = \Gamma(W,\mathcal{R})$.
\end{proof}

A second preparation of the proof of Theorem~\ref{mainthm1} 
concerns toric ambient modification.
We will always write $e_1, \ldots, e_n \in \ZZ^n$ 
for the canonical basis vectors and set 
$e_0 := -e_1  - \ldots - e_n$.
Moreover, we denote by $\Delta(n)$ the fan in $\ZZ^n$
consisting of all cones spanned by at most $n$ of 
the vectors $e_0, \ldots, e_n$ and by 
$\Delta'(n) \subseteq \Delta(n)$ 
the subfan consisting of all cone of dimension at 
most $n-1$.

\begin{lemma}
\label{lem:deltafan}
Consider nonzero vectors $v_1, \ldots, v_l \in \QQ^n$ 
contained in a maximal cone $\tau \in \Delta(n)$,
a cone $\sigma  \subseteq \QQ^n$ generated 
by some of the vectors $e_0,\ldots,e_n,v_1,\ldots,v_l$ 
and a cone $\delta  \in \Delta'(n)$.
Suppose that $\varrho := \delta \cap \sigma$ 
is one-dimensional and $\varrho \not\in \Delta'(n)$.
Then~$\varrho$ is contained in some facet of 
$\tau$.
\end{lemma}

\begin{proof}
We may assume that $\tau = \cone(e_1,\ldots, e_n)$ 
holds.
Replacing $\delta$ and $\sigma$ with suitable faces,
we may assume 
$\varrho^\circ = \delta^\circ \cap \sigma^\circ$.
The proof uses Gale duality and
we work in the notation of Section~\ref{sec:quadrics1}. 
Consider the matrix
$P := [e_0,\ldots, e_n,v_1,\ldots, v_l]$ 
and its Gale dual
\begin{eqnarray*}
Q
\ := \ 
[q_0,\ldots,q_{n+l}]
\ := \
\left[
\begin{array}{ccccccc}
0 & v_{11}&  \cdots &  v_{1n} & -1  &        & 0
\\
\vdots & \vdots &  &    \vdots     &     & \ddots & 
\\
0 & v_{l1} & \cdots & v_{ln} & 0 & \ & -1
\\
1 & 1 & \cdots &  1 &  0 & \cdots & 0
\end{array}
\right].
\end{eqnarray*}
Set $r := n+l$, let $e'_0, \ldots, e'_r$ denote the 
canonical basis vectors of $\ZZ^{r+1}$ and 
$\gamma := \QQ_{\ge 0}^{r+1}$ the positive orthant.
Then there are faces $\gamma_1,\gamma_2 \preceq \gamma$
such that for the corresponding dual faces 
$\gamma_i^*$ we have
$$ 
P(\gamma_1^*)  \  = \ \delta,
\qquad
P(\gamma_2^*) \ = \ \sigma,
\qquad
P(\gamma_1^*)^\circ \cap P(\gamma_2^*)^\circ \ \ne \ \emptyset.
$$
For some $n +1 \le j \le r$ we have $e'_j \in \gamma_2^*$ and 
we may assume that $\gamma_1^*$ is generated by at most
$n-1$ of the
vectors $e'_0, \ldots, e'_n$.
The latter implies $e'_{n+1}, \ldots, e'_{n+l} \in \gamma_1$.
Let $f = Q^*(u)$ be a separating linear 
form for $\gamma_1$ and $\gamma_2$. 
Then $f_{\vert \gamma_1} \ge 0$ implies
$$ 
u(q_{n+1}), \ldots, u(q_{n+l}) \ \ge \ 0,
\qquad
u(q_0) \ \ge \ u(q_1), \ldots, u(q_n). 
$$
Note that we must have $f(e'_j) = u(q_{j})> 0$, because $e'_j$ does 
not lie in $\gamma_2$.
Let $\tau_1, \tau_2 \preceq \gamma$ be the maximal 
faces with $f_{\vert \tau_1} \ge 0$ and 
$f_{\vert \tau_2} \le 0$. 
Then $f$ separates $\tau_1,\tau_2$ 
and $\tau_i^* \subseteq \gamma_i^*$ holds.
We conclude 
$$
\emptyset
\ \ne \ 
P(\tau_1^*)^\circ \cap P(\tau_2^*)^\circ 
\ \subseteq \ 
P(\tau_1^*) \cap  P(\tau_2^*)
\ \subseteq \
P(\gamma_1^*) \cap  P(\gamma_2^*)
\ = \ 
\varrho.
$$
Since $e'_j \not\in \tau_2$ holds, we obtain 
$\tau_2^* \ne \{0\}$ and thus  
$0 \not\in P(\tau_2^*)^\circ$.
Together with the displayed line this gives  
$P(\tau_1^*) \cap P(\tau_2^*) = \varrho$.
Since at least two of $e_0', \ldots, e_n'$
lie in $\gamma_1$, 
we obtain $e_0' \in \tau_1$ 
and thus 
$$
\varrho 
\ \subseteq \ 
P(\tau_1^*)
\ \subseteq \ 
\cone(e_1,\ldots, e_n).
$$
\end{proof}

\begin{lemma}
\label{lem:pnblowup}
For $n \in \ZZ_{\ge 1}$ consider $\Delta'(n)$ 
and let $b_1,\ldots, b_l \in \QQ^{n}$ be 
pairwise different primitive vectors lying on the 
support of $\Delta'(n)$ but not on its rays.
Denote by $\sigma_j \in \Delta'(n)$ the 
minimal cone with $b_j \in \sigma_j$ and 
write
$$ 
b_j
\ = \ 
a_{0j}e_0 + \ldots + a_{{n}j}e_{{n}},
\quad\text{where }
a_{ij} > 0 \text{ if } e_i \in \sigma_j, 
\
a_{ij} = 0 \text{ if } e_i \not\in \sigma_j.
$$
Then, for $P := [e_0,\ldots,e_{n}]$
and $B := [b_1,\ldots,b_l]$,
the matrix $A := (a_{ij})$ is a weak $B$-lifting 
with respect to $P$.
The lift of $h_1 = T_0 + \ldots + T_n$
in the sense of Lemma~\ref{lem:befreiung} is 
given by
\begin{eqnarray*}
h_2 
& = &
T_0S_1^{a_{01}} \cdots S_l^{a_{0l}}
\ + \ 
\ldots
\ + \ 
T_nS_1^{a_{n1}} \cdots S_l^{a_{nl}}.
\end{eqnarray*}
Moreover, the variables $T_0,\ldots T_n,S_1, \ldots,S_l$
define pairwise nonassociated prime elements in 
$\KK[T_0,\ldots T_n,S_1, \ldots,S_l]/\bangle{h_2}$ 
if and only if the vectors $b_1,\ldots, b_l$ 
lie in a common cone of $\Delta(n)$.  
\end{lemma}

\begin{proof}
Only the last sentence needs some explanation.
The fact that $b_1,\ldots, b_l$ lie in a common 
cone of $\Delta(n)$ is equivalent to the fact 
that there is a term of $h_2$ not depending on 
$S_1,\ldots,S_l$, and, moreover, for every $k$ there 
is a further term of $h_2$ not depending on~$S_k$.
Now, Lemma~\ref{lem:befreiung}~(iv) 
gives the desired characterization.
\end{proof}

\begin{proof}[Proof of Theorem~\ref{mainthm1}]
We may assume that $X = V(g_1)\subseteq \PP_{r}$
holds with a polynomial 
$g_1 = T_0T_1+\ldots +T_{r-1}T_{r}$,
where we replace the last term with 
$T_r^2$ in the case of an even $r$,
and $\KK^*$ acts linearly with weights 
$\zeta_0, \ldots, \zeta_{r}$, where  $|\zeta_r|$ is minimal 
among all $|\zeta_i|$, see~\cite[Prop.~III.2.4.7]{ADHL}.

The first step is to determine the normalized 
Chow quotient of the $\KK^*$-action on $X$.
As observed in Proposition~\ref{thm:chow2GIT}, 
the Chow quotient $X \quotchow \KK^*$
is canonically embedded into the Chow 
quotient of $\PP_{r}$ by the $\KK^*$-action.
To determine the latter, consider the extended 
weight matrix 
\begin{eqnarray*}
Q
& := &
\left[\begin{array}{rrr}
           \zeta_0&\ldots&\zeta_{r}\\
           1&\ldots&1
          \end{array}
        \right]
\end{eqnarray*}
and let $P$ be a Gale dual matrix. Then, according to 
Proposition~\ref{prop:chow2GIT}, the normalized 
Chow quotient of the $\KK^*$-action on $\PP_{r}$ is 
the toric variety $Z$ having the 
Gelfand-Kapranov-Zelevinsky decomposition 
$\Sigma$ defined by the columns of $P$ as its fan.
Moreover, by Proposition~\ref{thm:chow2GIT}, the Chow
quotient of the $\KK^*$-action on $X$ has the same 
normalization as the closure
$$
Y
\ = \ 
\overline{(X\cap \TT^{r})\,/\,\KK^*}
\ \subseteq \ 
Z. 
$$

The second step is to determine a weak tropical 
resolution of $Y\subseteq Z$. 
For this we first need $\trop(Y\cap T_Z)$. 
Let $\mu_0,\ldots,\mu_n \in \ZZ^{r+1}$ be the vertices
of the Newton polytope $g_1$ and consider the matrix
$P_{\mathrm{gr}}$ with the rows $\mu_i-\mu_0$, where 
$i=1,\ldots,n$.
Then we obtain a commutative diagram with exact rows
$$
\xymatrix{
0
\ar[r]
&
\QQ^2
\ar[r]
\ar[d]
&
\QQ^{r+1}
\ar[r]^{P}
\ar@{=}[d]
&
\QQ^{r-1}
\ar[r]
\ar[d]^{\Pi}
&
0
\\
0
\ar[r]
&
\QQ^{r+1-n}
\ar[r]
&
\QQ^{r+1}
\ar[r]_{P_{\mathrm{gr}}}
&
\QQ^n
\ar[r]
&
0
}
$$
Note that $g_1$ equals $T^{\mu_0}$ times the pullback 
of the polynomial 
$h_1 := 1 + S_1 + \ldots + S_n$ under the homomorphism 
of tori $\TT^{r} \to \TT^n$ defined by $P_{\mathrm{gr}}$.
The tropical variety of $V(h_1) \subseteq \TT^n$ is the
support of the fan $\Delta'(n)$ and thus we have 
$$ 
\trop(Y \cap T_Z)
\ = \ 
\Pi^{-1}(\trop(V(h_1)))
\ = \ 
\Pi^{-1}( \vert \Delta'(n) \vert ).
$$
We endow $\trop(Y \cap T_Z)$ with the natural fan 
structure lifting $\Delta'(n)$; note that the cones 
are in general not pointed.
By definition, the weak tropical resolution $Y'$ 
of $Y$ is the closure of $Y \cap T_Z$ in the 
toric variety $Z'$ with the coarsest common 
refinement $\Sigma' := \Sigma \sqcap \trop(Y \cap T_Z)$ 
as its fan.
 
In the third step, we pass to $Y' \quotnull T_{Y'}$, 
where $T_{Y'}$ is the kernel of the homomorphism 
of tori $T_Z \to \TT^n$ defined by $\Pi$.
By Construction~\ref{constr:nullquot}, the quotient
$Y' \quotnull T_{Y'}$ is the closure of the image 
of $Y \cap T_Z$ under $T_Z \to \TT^n$ in
the toric variety $Z' \quotnull T_{Y'}$
associated to the describing fan in $\ZZ^n$
having as maximal cones the rays $\Pi(\varrho)$, 
where~$\varrho$ runs through the rays of $\Sigma'$.

\smallskip

\noindent
{\em Claim. } 
For every ray $\varrho \in \Sigma'$ there is a facet 
of $\cone(e_0,\ldots, e_{n-1})$ containing the image 
$b := \Pi(\varrho) \in \QQ^n$. 

\smallskip

\noindent
Indeed, since every cone of $\trop(Y \cap T_Z)$ is 
saturated with respect to $\Pi$,
we have $\Pi(\varrho) = \Pi(\sigma) \cap \delta$ for 
some $\sigma \in \Sigma$ and $\delta \in \Delta'(n)$. 
The image $\Pi(\sigma)$ is a cone spanned by some 
$e_i$ and some images $v_j := \Pi(\nu_j)$, where 
$\nu_j$ are the primitive generators of the rays of 
$\Sigma$ different from columns $p_i$ of $P$.
Proposition~\ref{cor:gkzrays} yields presentations
$$
\nu_j
\ = \ 
\sum_{i=0}^{r-1}\alpha_{ij}p_i
\qquad
\text{ with certain }\alpha_{ij} \ge 0.
$$
Hence we obtain $v_j \in \cone(e_0,\ldots, e_{n-1})$.
Lemma~\ref{lem:deltafan} then shows that 
$\Pi(\varrho)$ lies in some facet of $\cone(e_0,\ldots e_{n-1})$
and the claim is verified.

\smallskip

Finally, in the fourth step, we show that 
$Y' \quotnull T_{Y'}$ is normal and
has finitely generated Cox ring;
by Proposition~\ref{prop:mdscrit} this will 
complete the proof.
First note that we have the toric modification 
$Z' \quotnull T_{Y'} \to W$, where 
$W \subseteq \PP_n$ is the open toric subset corresponding 
to the subfan $\Delta'(n)$ of $\Delta(n)$.
Moreover, $Y' \quotnull T_{Y'}$ is the proper transform 
under $Z' \quotnull T_{Y'} \to W$ of the closure
of $V(h_1) \subseteq \TT^n$ in $W$. 
The claim just verified and 
Lemma~\ref{lem:pnblowup} ensure that we may apply
Proposition~\ref{prop:K1primcrit} and 
Corollary~\ref{cor:ambientblowup}.
In particular, we see that $Y' \quotnull T_{Y'}$
is normal with finitely generated Cox ring.
\end{proof}

\begin{example}
Consider the quadric $X=V(T_0T_1+ \ldots+ T_6T_7)\subseteq \PP_7$ 
and the action of $\KK^*$ on $\PP_7$ given by
\begin{eqnarray*}
 t\cdot [x_0,\ldots,x_7]
& := &
[t^{-3}x_0,t^{3}x_1,t^{-3}x_2,t^{3}x_3,t^{-2}x_4,t^{2}x_5,t^{-1}x_6,tx_7].
\end{eqnarray*}
Theorem~\ref{mainthm2} and its proof do not apply 
to this case, 
because only two weights $\zeta_i$ have minimal 
absolute value.
The way through the weak toric resolution $Y'$ 
as gone in the proof of Theorem~\ref{mainthm1} 
produces a quotient $Y' \quotnull T_{Y'}$ 
embedded into the toric variety with fan obtained 
by subdividing $\Delta(3)$ at $(0,-1,-1)$.
\end{example}

\end{document}